\documentclass[11pt,a4paper]{article}
\usepackage{epsf,epsfig,amsfonts,amsgen,amsmath,amssymb,amstext,amsbsy,amsopn,amsthm,mathrsfs}
\usepackage{color}
\usepackage{graphicx}
\usepackage{caption}
\usepackage{epstopdf}
\usepackage{algorithmic}
\usepackage{algorithm}
\usepackage{float}
\usepackage{tikz,multirow}
\usepackage{diagbox}
\allowdisplaybreaks
\usetikzlibrary{decorations.markings,decorations.pathreplacing}
\tikzstyle{vertex}=[circle, draw, inner sep=0pt, minimum size=4pt]

\definecolor{DarkGreen}{rgb}{0.2, 0.6, 0.3}

\newtheorem{theorem}{Theorem}

\newtheorem{lemma}{Lemma}
\newtheorem{cor}{Corollary}

\theoremstyle{definition}

\newtheorem{problem}{Problem}

\begin{document}
\title
{\bf\Large Graphs with minimum degree-based entropy\thanks{The research was supported by the National Natural Science Foundation of China (Nos.~12071370, 12131013 and U1803263).}}
\date{}
\author{\small Yanni Dong${}^{1,2}$, Maximilien Gadouleau${}^{3}$, Pengfei Wan${}^{4,}$\thanks{Corresponding author. E-mail address: pfwan\_yulinu@outlook.com (P. Wan).} , Shenggui Zhang${}^{1,2}$
 \\[2mm]
\small ${}^1 $School of Mathematics and Statistics,\\
\small Northwestern Polytechnical University, Xi'an, Shaanxi 710129, P.R.~China\\
\small ${}^2 $Xi'an-Budapest Joint Research Center for Combinatorics,\\
\small Northwestern Polytechnical University, Xi'an, Shaanxi 710129, P.R.~China\\
\small ${}^3 $Department of Computer Science,\\
\small Durham University, Durham DH1 3LE, UK\\
\small ${}^4 $School of Mathematics and Statistics,\\
\small Yulin University, Yulin, Shaanxi 719000, P.R.~China\\}
\maketitle
\begin{abstract}
%It is well-known that the first attempt to evaluate the complexity of a system quantitatively was due to Rashevsky in the 1955, based on the Shannon's entropy formula. The idea has been developed further, and various graph invariants are used to construct entropy-based measurements to quantify the structural complexity of networks, among which the degree-based graph entropy is one of the most popular measurements.
The degree-based entropy of a graph is defined as the Shannon entropy based on the information functional that associates the vertices of the graph with the corresponding degrees. In this paper, we study extremal problems of finding the graphs attaining the minimum degree-based graph entropy among graphs and bipartite graphs with a given number of vertices and edges. We characterize the unique extremal graph achieving the minimum value among graphs with a given number of vertices and edges and present a lower bound for the degree-based entropy of bipartite graphs and characterize all the extremal graphs which achieve the lower bound. This implies the known result due to Cao et al. (2014) that the star attains the minimum value of the degree-based entropy among trees with a given number of vertices.

\medskip
\noindent {\bf Keywords:} Complexity measure; Graph entropy; Bipartite graph \\
\noindent {\bf Mathematics Subject Classification:} 05C99, 94A17, 68R10
\smallskip

\end{abstract}

\section{Introduction}

All graphs considered in this paper are simple, finite and undirected. The logarithms here are base $2$. We use the convention that $0 \log 0=0$.

A system that consists of a large number of microscopic components interacting with each other appears to be complex always \cite{Simon}. In order to describe the complexity or the information of a system, various graph entropies were introduced (refer to \cite{Bonchev,Dehmer2016,Dehmer2011} for reviewing). Reshevsky is the pioneer to quantify  the complexity of a system by the so-called topological information content, which is the earliest graph entropy  measurement \cite{Rashevsky}. Since then, many graph entropies based on various graph invariants, such as the number of matchings, independent sets and spanning forests, and the Randi\'c index \cite{Cao2017,Cao,Wan2020,Wan2019}, have been studied. Compared with other entropies, Cao et al. \cite{Cao} proposed an easily computable graph entropy called degree-based entropy, which has been proved useful in information theory, complexity networks and mathematical chemistry \cite{Bollobas2004,Bollobas2012}.

Let $G$ be a graph vertex set $V=\{v_{1},v_{2},\ldots,v_{n}\}$, and let $d_{G}(v_i)$ be the degree of $v_i$ of $G$. The {\em degree-based graph entropy} of $G$ is defined as
\begin{align}\label{eq1}
I_{d}(G)&=-\sum_{i=1}^{n}\frac{d_{G}(v_i)}{\sum_{j=1}^n d_{G}(v_j)}\log\frac{d_{G}(v_i)}{\sum_{j=1}^n d_{G}(v_j)}.
\end{align}

One of the attracting and challeging  problem in the theoretic study of graph entropy is giving extremal values of entropies and determing the correspoding graphs. Cao et al. \cite{Cao} proved some extremal values for the degree-based entropy of some certain families of graphs, such as trees, unicyclic graphs and bicyclic graphs, and proposed conjectures to determine extremal values. Shortly afterwards Ili\'c \cite{Ilic} proved one part of the conjectures. Ghalavand et al. \cite{Ghalavand} applied majorization technique to extend some known results on the extremal values of the degree-based entropy of trees, unicyclic and bicyclic graphs. Dong et al. \cite{Dong} obtained the maximum value of the degree-based entropy among bipartite graphs with a given number of  vertices and edges by characterizing corresponding degree sequences. Motivated by these results on extremal values of degree-based entropy, we study the problem of minimum values of the degree-based entropy among graphs and bipartite graphs with a given number of vertices and edges.

 A graph (resp. bipartite graph) with $n$ vertices and $m$ edges is referred to  as an {\em $(n,m)$-graph} (resp. {\em $(n,m)$-bipartite graph}). A simple graph in which each pair of distinct vertices is joined by an edge is called a {\em complete graph}. Up to isomorphism, there is just one complete graph with $n$ vertices, denoted by $K_n$. A {\em complete bipartite graph} is a simple bipartite graph with bipartition $(X,Y)$ in which each vertex of $X$ is joined to each vertex of $Y$; if $|X|=s$ and $|Y|=t$, such a graph is denoted by $K_{s,t}$. Let $G=(V,E)$ be a simple graph.  The {\em complement} $\overline{G}$ of $G$ is the simple graph whose vertex set is $V(G)$ and whose two distinct vertices are adjacent if and only if they are not adjacent in $G$. The {\em union} of simple graphs $G$ and $H$, denoted by  $G\cup H$, is the graph with the vertex set $V(G)\cup V(H)$ and the edge set $E(G)\cup E(H)$. Let $m$ and $k$ be two integers, and let $k^{*} =\max\{k: {k \choose 2}\leq m\}$ and $t^{*}=m-{{k^{*}}\choose 2}$. Note that the size of a maximum clique of a graph with $m$ edges cannot be more than $k^{*}$. The graph $K(k,t)$ is obtained by adding a vertex adjacent to $t$ vertices of $K_{k}$. Let
$$\sigma(x)=
\left\{\begin{array}{lc}
0,&\;\mathrm{if}\hspace{0.33em}x=0;\hfill\\
1,&\;\mathrm{otherwise.}\hspace{0.33em}\hfill
\end{array}
\right.$$

We characterize the extremal graph achieving the minimum value among $(n,m)$-graphs.

\begin{theorem}\label{th4}
Let $G$ be an $(n,m)$-graph with $n\geq 2$, $1\leq m\leq {n \choose 2}$. Then
\begin{align*}
I_{d}(G)\geq \log(2m)-\frac{t^{*}k^{*}\log k^{*}+(k^{*}-t^{*})(k^{*}-1)\log(k^{*}-1)+t^{*}\log t^{*}}{2m},
\end{align*}
the equality holds if and only if $G \cong K(k^{*},t^{*})\cup \overline{K}_{n-k^{*}-\sigma(t^{*})}.$
\end{theorem}

We will use the following theorem that characterizes all the extremal graphs achieving a lower bound among bipartite graphs with a given number of vertices and edges to prove the above theorem.

\begin{theorem}\label{th3}
Let $G$ be an $(n,m)$-bipartite graph with $n\geq 2$ and $1\leq m\leq \lceil\frac{n}{2}\rceil\lfloor\frac{n}{2}\rfloor$. Then $I_{d}(G)\geq 1+\log{\sqrt{m}},$ the equality holds in the inequality if and only if $G\cong K_{q,b}\cup \overline{K}_{n-q-b}$, where $q$ and $b$ satisfy $qb=m$ and $q+b\leq n$.
\end{theorem}

The following result due to Cao et al. \cite{Cao} is immediately implied by Theorem \ref{th3}.

\begin{cor}[\cite{Cao}]\label{co2}
Let $T$ be a tree on $n$ vertices. Then $I_{d}(T)\geq 1+\log{\sqrt{n-1}}$, the equality holds in the inequality if and only if $T\cong K_{1,n-1}$.
\end{cor}

The rest of this paper is organized as follows. In Section 2, we give some preliminary results that will be used later. In Section 3, we give the proofs of our main results.  We conclude this paper with a problem and some remarks.

\section{Preliminaries}

Let $G=(V,E)$ be a graph, and let $a_i$ be the number of vertices with degree $d_i$ in a set $S \subseteq V(G)$ for $i=1,2,\ldots,t+1$. Denote by $D(S)=[d_{1}^{a_1},d_{2}^{a_{2}},\ldots,d_{t}^{a_t}]$ the degree sequence of $S$ in which $d_{1}>d_{2}>\cdots >d_{t+1}=0$ and $a_1+a_2+\cdots+a_{t+1}=|S|$. By $D(G)$ denote the degree sequence of $V(G)$. Every graph with the degree sequence $D$ is a {\em realization} of $D$. A degree sequence is {\em unigraphic} if all its realizations are isomorphic. We use $\Delta(S)$ and $\Delta (G)$ to denote the maximum degree of vertices in a set of vertices $S$ and a graph $G$, respectively. The size of a maximum clique in $G$ is denoted by $\omega(G)$.

Let $G$ be a graph with $m$ edges, and let $V=\{v_{1},v_{2},\ldots,v_{n}\}$ be the vertex set of $G$. By the hand shaking lemma,
\begin{align*}
 \sum_{j=1}^n d_{G}(v_j)=2m.
\end{align*}
From Equality (\ref{eq1}), we infer
\begin{align*}
I_{d}(G)&=\log \sum_{j=1}^n d_{G}(v_j)-\frac{1}{\sum_{j=1}^n d_{G}(v_j)}\sum_{i=1}^n d_{G}(v_i)\log d_{G}(v_i)\\
        &=\log (2m)-\frac{1}{2m}\sum_{i=1}^n d_{G}(v_i)\log d_{G}(v_i).
\end{align*}
We define a function
\begin{align*}\label{eq2}
h_{d}(G)=\sum_{i=1}^n d_{G}(v_i)\log d_{G}(v_i).
\end{align*}
Equation (\ref{eq1}) implies that $\min\limits_{G \in \mathcal{G}} I_{d}(G)=\log (2m)-\frac{1}{2m}\max\limits_{G \in \mathcal{G}} h_{d}(G)$ for a certain family of graphs $\mathcal{G}$.

Let us introduce the concept of majorization as a useful tool to solve some inequality problems \cite{Matshall}.
Let $A=[a_1,a_2,\ldots,a_n]$ and $B=[b_1,b_2,\ldots,b_n]$ be integer sequences of length $n$. We define that $A$ {\em majorizes} $B$, denoted by $A\succeq B$, if
\begin{align*}
\sum_{i=1}^s a_i\geq \sum_{i=1}^s b_i,~~ s=1,2,\ldots,n-1
\end{align*}
and
\begin{align*}
\sum_{i=1}^n a_i= \sum_{i=1}^n b_i.
\end{align*}
The majorization is {\em strict}, and is denoted by $A\succ B$, if at least one of the inequalities is strict. Let $c_1,c_2,\ldots,c_n$ be integers satisfying $c_1\geq c_2\geq\cdots \geq c_n\geq 0$.

Ghalavand et al. \cite{Ghalavand} extended some known results about the minimum and maximum values of the degree-based entropy by the following theorem.

\begin{theorem}[\cite{Ghalavand}]\label{le1}
Let $G$ and $G'$ be two $(n,m)$-graphs. If $D(G)\succeq D(G')$ , then $I_{d}(G)\leq I_{d}(G')$, the equality holds in this inequality if and only if $D(G)=D(G')$.
\end{theorem}

Threshold graphs have been defined by Chv\'atal and Hammer \cite{Chvatal1977} as follows.
A graph $G$ is called a {\em threshold graph} if there exist non-negative real numbers $w_{1},w_{2},\ldots,w_n$ associated with the vertex set $\{v_1,v_2,\ldots,v_n\}$ of $G$ and $t$ such that for any two distinct vertices $v_i$ and $v_j$, $v_i$ and $v_j$ are adjacent if and only if $w_i+w_j>t$. Let $\mathcal{C}$ be a collection of sets. The collection $\mathcal{C}$ is {\em nested} if for every two sets in $\mathcal{C}$, one is a subset of the other. Some basic characterizations of threshold graphs are given by the following theorem.

\begin{theorem}[\cite{Chvatal1977}]\label{th5}
Let $G=(V,E)$ be a graph. Then the following are equivalent:
\begin{description}
  \item[(a)] the graph $G$ is a threshold graph;
  \item[(b)] there are no four distinct vertices $u,v,w,x \in V(G)$ such that $uv,wx\in E(G)$ and $uw,vx\notin E(G)$;
  \item[(c)] the graph $G$ is a split graph $G(K,S)$ (i.e., its vertices can be partitioned into a clique $K$ and a stable set $S$) and the neighborhoods of the vertices of $S$ are nested.
\end{description}
\end{theorem}

The family of difference graphs was introduced by Hammer et al. \cite{Hammer1990}.
A graph $G$ is said to be a {\em difference graph} if there exist real numbers $w_{1},w_{2},\ldots,w_n$ associated with the vertex set $\{v_1,v_2,\ldots,v_n\}$ of $G$ and a positive real number $t$ such that
\begin{description}
  \item[(a)] $|w_i|<t$ for $i=1,2,\ldots,n$;
  \item[(b)] two distinct vertices $v_i$ and $v_j$ are adjacent if and only if $|w_i-w_j|\geq t$.
\end{description}

The following result is due to Mahadev and Peled \cite{Mahadev}.

\begin{theorem}[\cite{Mahadev}]\label{le5}
Let $G$ be a bipartite graph with bipartition $(X,Y)$. Then the following conditions are equivalent:
\begin{description}
  \item[(a)] the graph $G$ is a difference graph;
  \item[(b)] there are no $x_1,x_2 \in X$ and $y_1,y_2\in Y$ such that $x_1y_1,x_2y_2 \in E(G)$ and $x_1y_2,x_2y_1 \notin E(G)$;
  \item[(c)] every induced subgraph without isolated vertices has on each side of the bipartition a domination vertex, that is, a vertex is adjacent to all the vertices on the other side of the bipartition.
\end{description}
\end{theorem}

The {\em conjugate sequence} of a sequence $C=[c_{1},c_{2},\ldots,c_{n}]$ is the sequence $C^{*}=[c^{*}_1,$ $c^{*}_2,\ldots,c^{*}_n]$ in which $c^{*}_{i}=|\{j:c_{j}\geq i\}|$. The following result is one characterization of difference graphs by the degree sequences.

\begin{theorem}[\cite{Mahadev}]\label{le6}
A pair of  non-negative and non-increasing integer sequences $D(X)=[d_{G}(x_1),d_{G}(x_2),\ldots,d_{G}({x_{|X|}})]$ and $D(Y)=[d_{G}(y_1),d_{G}(y_2),\ldots,d_{G}({y_{|Y|}})]$ with $d_{G}(y_1)\leq|X|$ is a pair of degree sequences of a difference graph $G$ with bipartition $(X,Y)$ if and only if $D(X)=D^{*}(Y)$, where $D^{*}(Y)$ is the conjugate of $D(Y)$.
\end{theorem}

We need the next lemma to prove Theorem \ref{th4}.

\begin{lemma}\label{le8}
Let $G$ be an $(n,m)$-graph with $n\geq 2$ and $1\leq m\leq {n \choose 2}$. If $I_{d}(G)$ attains the minimum value among $(n,m)$-graphs, then $G$ is a threshold graph.
\end{lemma}
\begin{proof}
Suppose that $G$ is not a threshold graph. By Theorem \ref{th5} (b), there are four vertices $u,v,w,x\in V(G)$ such that $uv, wx\in E(G)$ and $ux, vw \notin E(G)$.  We may assume without loss of generality that $d_{G}(v)\geq d_{G}(x)$. Set $G'=G-wx+wv$. Hence $D(G')\succ D(G)$. By Theorem \ref{le1}, we have $I_{d}(G')< I_{d}(G)$, a contradiction.
\end{proof}

To prove Theorem \ref{th3}, we need the next lemma.

\begin{lemma}\label{le7}
Let $G$ be an $(n,m)$-bipartite graph with $n\geq 2$ and $1\leq m\leq \lceil\frac{n}{2}\rceil\lfloor\frac{n}{2}\rfloor$. If $I_{d}(G)$ attains the minimum value among $(n,m)$-bipartite graphs, then $G$ is a difference graph.
\end{lemma}
\begin{proof}
Let $(X,Y)$ be the bipartition of $G$. Suppose that $G$ is not a difference graph. By Theorem \ref{le5}(b), there are four vertices $x_1,x_2\in X$ and $y_1,y_2\in Y$  such that $x_1y_1,x_2y_2 \in E(G)$ and $x_1y_2,x_2y_1 \notin E(G)$.  We may assume without loss of generality that $d_{G}(x_1)\geq d_{G}(x_2)$. Set $G'=G-x_2y_2+x_1y_2$. Hence $D(G')\succ D(G)$. By Theorem \ref{le1}, we have $I_{d}(G')< I_{d}(G)$, a contradiction.
\end{proof}

\section{Proofs}

\begin{proof}[Proof of Theorem \ref{th4}]
Suppose that $G^{*}$ is an $(n,m)$-graph such that $I_{d}(G^{*})$ attains the minimum value among graphs with $n$ vertices and $m$ edges. It follows from Lemma \ref{le8} that $G^{*}$ is a threshold graph. And by Theorem \ref{th5} (c), the graph $G^{*}$ is a split graph $G(K,S)$ (i.e., its vertices can be partitioned into a clique $K$ and a stable set $S$) and the neighborhoods of the vertices of $S$ are nested.
If $K$ is not the maximum clique, then there exists a vertex $v\in S$ in the maximum clique.
Since $S$ is the stable set, there is at most one vertex $v\in S$ in the maximum clique. This implies that vertices of $G^{*}$ can be partitioned into a clique $K'=K\cup \{v\}$ and a stable set $S'=S\setminus \{v\}$. So we may assume $K$ is the maximum clique.
Let $k=|K|$ and $t=m-{|K| \choose 2}$.   Then we have $k=\omega(G^{*})$. It follows from $m\geq 1$ that $k\geq 2$. Since $m-{k\choose 2}\geq 0$, it is easy to get that $k\leq k^{*}$. Thus $2\leq k\leq k^{*}$.

Denote by $N(S)$ be the set of neighbors of vertices in $S$, and $H$ the subgraph of $G$ induced by all edges between $N(S)$ and $S$. So $H$ is a bipartite graph. Let $(X,Y)$ be the bipartition of $H$ in which $X=\{x_1,x_2,\ldots,x_{|X|}\} \subseteq K$ and $Y=\{y_1,y_2,\ldots,y_{|Y|}\} \subseteq S$.
If $m={k\choose 2}$ (i.e., $k= k^{*}$ and $t=t^{*}=0$), then $G^{*}\cong K_{k^{*}}\cup \overline{K}_{n-k^{*}}$.
So we consider $m>{k\choose 2}$ (i.e., $t\geq 1$) in the following. For $\sum_{j=1}^{|X|}{z_j}=t\geq |X|$ and $k\geq 2$,
we analyze conditional extremums of the following function
\begin{align*}
f({z_1},z_2,\ldots,{z_{|X|}})=\sum_{j=1}^{|X|}{(z_j+k-1)}\log {(z_j+k-1)}-\sum_{j=1}^{|X|}{z_j}\log {z_j}
\end{align*}
and corresponding Lagrangian function
\begin{align*}
L({z_1},z_2,\ldots,z_{|X|},\lambda)=f({z_1},z_2,\ldots,{z_{|X|}})+\lambda\left(\sum_{j=1}^{|X|}{z_j}-t\right)
\end{align*}
with the additional boundary condition ${z_j}\geq 1$. The function is well-defined and differentiable with respect to each of its arguments on the closed region, so the extremal values can be either critical points or on the boundary. By direct calculation, we obtain
\begin{align*}
\frac{\partial}{\partial \lambda}L=\sum_{j=1}^{|X|}{z_j}-t=0
\end{align*}
and
\begin{align*}
\frac{\partial}{\partial{z_j}}L=\log {(z_j+k-1)}-\log{z_j}+\lambda=0
\end{align*}
for $j=1,2,\ldots,|X|$.
It follows from this set of equations that the unique critical point satisfies ${z_1}=z_2=\cdots={z_{|X|}}=\frac{t}{|X|}\geq 1$. Using the second order conditions, this is a local maximum, as
\begin{align*}
\frac{\partial^2}{\partial{z^{2}_j}}L=\frac{1}{(z_j+k-1)\ln 2}-\frac{1}{z_j\ln2}<0,
\end{align*}
where $k\geq 2$.

For variables $d_{H}(x_1),d_{H}(x_2),\ldots,d_{H}(x_{|X|})$ satisfying $\sum_{j=1}^{|X|}{z_j}=t$, the function\\
$f(d_{H}(x_1),d_{H}(x_2),\ldots,d_{H}(x_{|X|}))$ attains the maximum value if and only if $d_{H}(x_j)=\frac{t}{|X|}$ for $j=1,2,\ldots,|X|$.
And by Theorem \ref{th3}, we have
\begin{align}
\nonumber
h_{d}(G^{*})=&h_{d}(K_k)+h_{d}(H)-|X|(k-1)\log(k-1)\\
\nonumber
         &-\sum_{j=1}^{|X|}d_{H}(x_j)\log d_{H}(x_j)\\
\nonumber
         &+\sum_{j=1}^{|X|}(d_{H}(x_j)+k-1)\log (d_{H}(x_j)+k-1)\\
\nonumber
     \leq&k(k-1)\log(k-1)+t\log t-|X|(k-1)\log(k-1)\\
         &-t\log \frac{t}{|X|}+|X|\left(\frac{t}{|X|}+k-1\right)\log\left(\frac{t}{|X|}+k-1\right),
\end{align}
the equality holds in this inequality if and only if $d_{H}(x_j)=\frac{t}{|X|}$ for $j=1,2,\ldots,|X|$.

Let $p(x)=-x(k-1)\log(k-1)-t\log \frac{t}{x}+x(\frac{t}{x}+k-1)\log(\frac{t}{x}+k-1)$ for $1 \leq x\leq t$. By calculation, we have
\begin{align*}
\frac{d}{dx}p=&(k-1)\log\left(1+\frac{t}{x(k-1)}\right)>0.
\end{align*}
This implies that $p(x)$ is rigidly monotonically increasing for $1 \leq x\leq t$. Let $t'=\min \{k-1,t\}$. Since $k=\omega(G^{*})$, $|X|\leq t^{'}$. Therefore, we have
\begin{align}
\nonumber
&k(k-1)\log(k-1)+t\log t-|X|(k-1)\log(k-1)\\
\nonumber
&-t\log \frac{t}{|X|}+|X|\left(\frac{t}{|X|}+k-1\right)\log\left(\frac{t}{|X|}+k-1\right)\\
\nonumber
      \leq&k(k-1)\log(k-1)+t\log t-t'(k-1)\log(k-1)\\
         &-t\log \frac{t}{t'}+t'\left(\frac{t}{t'}+k-1\right)\log\left(\frac{t}{t'}+k-1\right),
\end{align}
the equality holds in this inequality if and only if $|X|=t'$.

Let $g(y)=y(y-1)\log (y-1)-t'(y-1)\log(y-1)+t'(\frac{t}{t'}+y-1)\log(\frac{t}{t'}+y-1)$ for $2 \leq y\leq k^{*}$. By calculation, we have
\begin{align*}
\frac{d}{d y}g=&\frac{t'\ln\left(\frac{t'(y-1)+t}{t'}\right)+(2y-t'-1)\ln(y-1)+2y}{\ln 2}\\
  >&{t'\ln \left(\frac{t'+t}{t'}\right)}\\
  >&0.
\end{align*}
This implies that $g(y)$ is rigidly monotonically increasing for $2 \leq y\leq k^{*}$. So we have
\begin{align}
\nonumber
&k(k-1)\log(k-1)+t\log t-t'(k-1)\log(k-1)\\
\nonumber
         &-t\log \frac{t}{t'}+t'\left(\frac{t}{t'}+k-1\right)\log\left(\frac{t}{t'}+k-1\right)\\
\nonumber
         \leq& k^{*}(k^{*}-1)\log(k^{*}-1)+t\log t-t'(k^{*}-1)\log(k^{*}-1)\\
         &-t\log \frac{t}{t'}+t'\left(\frac{t}{t'}+k^{*}-1\right)\log\left(\frac{t}{t'}+k^{*}-1\right),
\end{align}
the equality holds in this inequality if and only if $k=k^{*}$.

If $t=t'$, then we have
\begin{align}
\nonumber
         & k^{*}(k^{*}-1)\log(k^{*}-1)+t\log t-t'(k^{*}-1)\log(k^{*}-1)\\
\nonumber
         &-t\log \frac{t}{t'}+t'\left(\frac{t}{t'}+k^{*}-1\right)\log\left(\frac{t}{t'}+k^{*}-1\right)\\
         =& t'k^{*}\log k^{*}+(k^{*}-t')(k^{*}-1)\log(k^{*}-1)+t'\log t'.
\end{align}
We prove $t'=t$ (i.e., $t<k$) when $k=k^{*}$. It follows from $k^{*} =\max\{k: {k \choose 2}\leq m\}$ that $k^{*}=\left\lfloor\frac{\sqrt{8m+1}+1}{2}\right\rfloor$. Suppose that $t\geq k$ for $k=\left\lfloor\frac{\sqrt{8m+1}+1}{2}\right\rfloor$. So we have $m-{k \choose 2}\geq k$ for $k=\left\lfloor\frac{\sqrt{8m+1}+1}{2}\right\rfloor$. This implies $k\leq\frac{\sqrt{8m+1}-1}{2}$ which contradicts $k=\left\lfloor\frac{\sqrt{8m+1}+1}{2}\right\rfloor$. So we have $t'=t$ when $k=k^{*}$.  It follows from $k=k^{*}$ that $t'=t=m-{k \choose 2}=m-{k^{*} \choose 2}=t^{*}$. We consider $m>{k\choose 2}$. This implies $\sigma({t^{*}})=1$.
And by (1), (2), (3), (4) and (5), we have
\begin{align*}
I_{d}(G)\geq \log(2m)-\frac{t^{*}k^{*}\log k^{*}+(k^{*}-t^{*})(k^{*}-1)\log(k^{*}-1)+t^{*}\log t^{*}}{2m},
\end{align*}
the equality holds in this inequality if and only if $$G \cong K(k^{*},t^{*})\cup \overline{K}_{n-k^{*}-1}.$$
\end{proof}

\begin{proof}[Proof of Theorem \ref{th3}]
Suppose that $G^{*}$ is an $(n,m)$-bipartite graph such that $I_{d}(G^{*})$ attains the minimum value among bipartite graphs with $n$ vertices and $m$ edges. Assume that $n\ge m+1$. Let $(X,Y)$ be the bipartition of $G^{*}$. We denote the vertices of $X$ and $Y$ by $x_1,x_2,\ldots,x_{|X|}$ and $y_1,y_2,\ldots,$ $y_{n-|X|}$, respectively. From Lemma \ref{le7}, $G^{*}$ is a difference graph. By Theorem \ref{le5} (c),  without isolated vertices, $G^{*}$ has a domination vertex on each side of the bipartition. Without loss of generality, we assume that $x_1$  is the domination vertex in $X$  and $x_1y_j\in E(G^{*})$  for $j=1,2,\ldots,b$. So we have $d_{G^{*}}(x_1)=b$. If $m=b$, then $G^{*}\cong K_{1,b}\cup \overline{K}_{n-b-1}$.

For $m> b$, we prove the  extremal graphs by induction on $n$.  For $n=2$, $G^{*}\cong K_{1,1}$.
For $n\geq 3$, set $G'=G^{*}-x_{1}y_1-x_{1}y_2-\cdots-x_1y_b$ and assume that $G'\cong K_{q-1,b}\cup \overline{K}_{n-q-b}$.

We have a recursive relation

\begin{small}
\begin{align*}
h_{d}(G^{*})=&\sum_{i=1}^{|X|}d_{G^{*}}(x_i)\log d_{G^{*}}(x_i)+\sum_{j=1}^{n-|X|}d_{G^{*}}(y_j)\log d_{G^{*}}(y_j)\\
        =&\sum_{i=1}^{|X|}d_{G^{*}}(x_i)\log d_{G^{*}}(x_i)+\sum_{j=1}^{b}d_{G^{*}}(y_j)\log d_{G^{*}}(y_j)\\
        =&\sum_{i=2}^{|X|}d_{G^{*}}(x_i)\log d_{G^{*}}(x_i)+\sum_{j=1}^{b}(d_{G^{*}}(y_j)-1)\log (d_{G^{*}}(y_j)-1)+\sum_{j=1}^{b}d_{G^{*}}(y_j)\log d_{G^{*}}(y_j)\\
        &-\sum_{j=1}^{b}(d_{G^{*}}(y_j)-1)\log (d_{G^{*}}(y_j)-1)+b\log b\\
        =&h_{d}(G')+\sum_{j=1}^{b}d_{G^{*}}(y_j)\log d_{G^{*}}(y_j)-\sum_{j=1}^{b}(d_{G^{*}}(y_j)-1)\log (d_{G^{*}}(y_j)-1)+b\log b.
\end{align*}
\end{small}

For $\sum_{j=1}^{b}{z_j}=\sum_{j=1}^{b}d_{G^{*}}(y_j)=m>b$,
we analyze conditional extremums of the following function
\begin{align*}
f({z_1},z_2,\ldots,{z_b})=\sum_{j=1}^{b}{z_j}\log {z_j}-\sum_{j=1}^{b}({z_j}-1)\log ({z_j}-1)
\end{align*}
and corresponding Lagrangian function
\begin{align*}
L({z_1},z_2,\ldots,{z_b},\lambda)=f({z_1},z_2,\ldots,{z_b})+\lambda\left(\sum_{j=1}^{b}{z_j}-m\right)
\end{align*}
with the additional boundary condition ${z_j}> 1$. The function is well-defined and differentiable with respect to each of its arguments on the closed region, so the extremal values can be either critical points or on the boundary. By direct calculation, we obtain
\begin{align*}
\frac{\partial}{\partial \lambda} L=\sum_{j=1}^{b}{z_j}-m=0
\end{align*}
and
\begin{align*}
\frac{\partial}{\partial {z_j}} L=\log {z_j}-\log({z_j}-1)+\lambda=0
\end{align*}
for $j=1,2,\ldots,b$.
It follows from this set of equations that the unique critical point satisfies ${z_1}=z_2=\cdots={z_b}=\frac{m}{b}> 1$. Furthermore, using the second order conditions, it follows that this is a local maximum, as

\begin{align*}
\frac{\partial^2}{\partial z^{2}_{j}}L=-\frac{1}{\ln 2{z_j}({z_j}-1)}<0.
\end{align*}

For variables $d_{G^{*}}(y_1),d_{G^{*}}(y_2),\ldots,d_{G^{*}}(y_b)$, the function $f(d_{G^{*}}(y_1),d_{G^{*}}(y_2),\ldots,d_{G^{*}}(y_b))$ attains the maximum value if and only if $d_{G^{*}}(y_j)=\frac{m}{b}$ for $j=1,2,\ldots,b$. So we have $D(Y)=[q^b]$ in which $q=\frac{m}{b}$. By Theorem \ref{le6}, we have $D(X)=D^{*}(Y)$, where $D^{*}(Y)$ is the conjugate of $D(Y)$. This implies that the degree sequence of $X$ is $D(X)=[b^{q}]$. Obviously, this pair of degree sequences is unigraphic, that is, $K_{q,b}\cup \overline{K}_{n-q-b}$ is the only realization of this pair of degree sequences up to isomorphism.
If $n\geq m+1$, then there exist two integers $b$ and $q$ such that $m=qb$ and $n\geq q+b$. By induction hypothesis, $G^{*}\cong K_{q,b}\cup \overline{K}_{n-q-b}$ for $n\geq m+1$. By calculation, we have
\begin{align*}
I_{d}(K_{q,b}\cup \overline{K}_{n-q-b})=&\log(2m)-\frac{1}{2m}(bq \log q+qb \log b)\\
                                      =&\log (2m)-\frac{1}{2m}(m \log m)\\
                                      =&1+\log{\sqrt{m}}.
\end{align*}
Therefore we have $I_{d}(G)\leq 1+\log{\sqrt{m}}$, the equality holds in the inequality if and only if $G\cong K_{q,b}\cup \overline{K}_{n-q-b}$, where $q$ and $b$ satisfy $qb=m$ and $q+b\leq n$.

Now we consider $n\leq m$ in the following. It is sufficient to prove $I_{d}(G^{*})> 1+\log{\sqrt{m}}$ if $G^{*}\cong K_{q,b}\cup \overline{K}_{n-q-b}$ does not hold for any integers $b$ and $q$. Contrarily, we assume that $I_{d}(G^{*})\leq 1+\log{\sqrt{m}}$. It follows that $I_{d}(G^{*})=I_{d}(G^{*}\cup \overline{K}_{m+1-n})\geq I_{d}(K_{1,m})=1+\log{\sqrt{m}}$. So we have $I_{d}(G^{*})= 1+\log{\sqrt{m}}$. This implies $I_{d}(G^{*}\cup \overline{K}_{m+1-n})$ attains the minimum value among $(m+1,m)$-bipartite graphs and $G^{*}\cup \overline{K}_{m+1-n}$ is not isomorphic to $K_{q,b}\cup \overline{K}_{n-q-b}$ for any integers $q$ and $b$, which contradicts that $G^{*}\cong I_{d}(K_{q,b}\cup \overline{K}_{n-q-b})$ for $n\geq m+1$.
\end{proof}
\section{Concluding remarks}

Among bipartite graphs with $n$ vertices and $m$ edges, we prove that $I_{d}(K_{q,b}\cup \overline{K}_{n-q-b})$ achieves the lower bound $1+\log{\sqrt{m}}$, where $qb=m$ and $q+b\leq n$. As an example, we consider the extremal results on bipartite with at most six vertices. Before we list the results, we define a special bipartite graph. Let $n$, $m$ and $b$ be integers with $n\geq 2$ and $\lceil\frac{n}{2}\rceil\lfloor\frac{n}{2}\rfloor \geq m\geq b\geq 1$. Set $q=\lfloor\frac{m}{b}\rfloor$, $r=m-b\lfloor\frac{m}{b}\rfloor$.
Let $B(n,m,b)$ be the $(n,m)$-bipartite graph with bipartition $(X,Y)$, where $X=\{x_1,x_2,\ldots,$ $x_{|X|}\}$, $Y=\{y_1,y_2,\ldots,y_{|Y|}\}$, $|X|\geq q+\sigma(r)$, $|Y|\geq b$, $x_iy_j \in E(B(n,m,b))$ for $i=1,2,\ldots, q$, $j=1,2,\ldots, b$ and $x_{q+1}y_s \in E(B(n,m,b))$ for $s=1,2,\ldots,r$. The bipartite graphs attaining the minimum value of degree-based entropy among $(n,m)$-bipartite graphs are listed in the following table for $n=2,3,\ldots,6$.

\begin{center}
{Table 1: The bipartite graphs with the minimum value of degree-based entropy.}\\
\medskip
\begin{tabular}{|c|c|c|c|c|c|}
  \hline
  \diagbox{$m$}{$n$} & 2 & 3 & 4 & 5 & 6 \\
  \hline
  1 & $K_{1,1}$ &  $K_{1,1}\cup \overline{K}_1$ & $K_{1,1}\cup \overline{K}_2$ & $K_{1,1}\cup\overline{K}_3$ & $K_{1,1}\cup\overline{K}_4$ \\
  \hline
  2 &  &  $K_{1,2}$ & $K_{1,2}\cup \overline{K}_1$ &  $K_{1,2}\cup \overline{K}_2$ &  $K_{1,2}\cup \overline{K}_3$ \\
  \hline
  3 &  &  &  $K_{1,3}$ &  $K_{1,3}\cup \overline{K}_1$ &  $K_{1,3}\cup \overline{K}_2$ \\
  \hline
  4 &  &  &  $K_{2,2}$ &  $K_{2,2}\cup \overline{K}_1$ and $K_{1,4}$&  $K_{2,2}\cup \overline{K}_2$ and $K_{1,4}\cup \overline{K}_1$ \\
  \hline
  5 &  &  &  & $B(5,5,3)$ & $K_{1,5}$ \\
  \hline
  6 &  &  &  & $K_{2,3}$ &  $K_{2,3}\cup \overline{K}_1$ \\
  \hline
  7 &  &  &  &  & $B(6,7,4)$ \\
  \hline
  8 &  &  &  &  & $K_{2,4}$ \\
  \hline
  9 &  &  &  &  & $K_{3,3}$ \\
  \hline
\end{tabular}
\end{center}
\medskip
By observation, we find $I_{d}(B(6,7,3))<I_{d}(B(6,7,4))$. Does $I_{d}(B(n,m,b))$ decrease as the variable $b$ increases for any given $n$ and $m$? The answer is no since $I_{d}(B(7,7,4))<I_{d}(B(7,7,3))<I_{d}(B(7,7,5))$. Let $\mathcal{B}(n,m,b)$ be the set of $(n,m)$-bipartite graphs satisfying that the maximum degree of one part is $b$. Does $G\cong B(n,m,b)$ hold if $I_{d}(G)$ attains the minimum value among all graphs in $\mathcal{B}(n,m,b)$? The answer is no since $I_{d}(B(7,10,4))>I_{d}(B(7,10,3))$ and $\mathcal{B}(7,10,3)\subseteq \mathcal{B}(7,10,4)$ up to isomorphism. So we pose a general problem as follows.

\begin{problem}
If there does not exist a complete bipartite graph $K_{q,b}$ satisfying $qb=m$ and $q+b\leq n$, then how to find the ones attaining the minimum degree-based entropy?
\end{problem}

\end{document}